\documentclass[a4paper,12pt]{amsart}

\makeindex

\usepackage{xcolor}
\usepackage[
colorlinks=true,
urlcolor=purple,
linkcolor=purple!87!black,
pdfborder={0 0 0}
]{hyperref}

\usepackage{amsfonts,graphics,amsmath,amsthm,amsfonts,amscd,amssymb,amsmath,latexsym,euscript, enumerate}
\usepackage{epsfig,url}
\usepackage{flafter}
\usepackage[all,cmtip,line]{xy}
\usepackage{array}
\usepackage[english]{babel}
\usepackage{overpic}
\usepackage{subfig}
\usepackage{multirow}
\usepackage{tikz-cd}
\usepackage{theoremref}

\usepackage{wrapfig}
\usepackage[shortlabels]{enumitem}
\setlist[enumerate, 1]{1\textsuperscript{o}}

\newtheorem{thm}{Theorem}[section]
\newtheorem{lem}[thm]{Lemma}
\newtheorem{prop}[thm]{Proposition}

\theoremstyle{corollary}
\newtheorem{coro}[thm]{Corollary}

\theoremstyle{definition} 
\newtheorem{defi}[thm]{Definition}
\newtheorem{definition-lemma}[thm]{Definition-Lemma}
\newtheorem{definitionThm}[thm]{Definition-Theorem}
\newtheorem{ex}[thm]{Example}

\theoremstyle{remark}
\newtheorem{rmk}[thm]{Remark}

\newtheorem*{ack}{Acknowledgments}

\newtheorem{notation}[thm]{Notation}
\numberwithin{equation}{section}

\newcommand{\C}{\mathbb{C}}

\newcommand{\Q}{\mathbb{Q}}

\newcommand{\m}{\mathfrak{m}}

\DeclareRobustCommand{\O}{\mathcal{O}}

\DeclareMathOperator{\Hom}{\mathrm{Hom}}
\DeclareMathOperator{\Ext}{\mathrm{Ext}}

\def\Spec{\operatorname{Spec}}

\def\Supp{\operatorname{Supp}}

\def\codim{\operatorname{codim}}

\newcommand{\floor}[1]{\left\lfloor #1 \right\rfloor}
\newcommand{\ceil}[1]{\left\lceil #1 \right \rceil}
\let\oldframe\frame
\renewcommand\frame[1][allowframebreaks]{\oldframe[#1]}

\title[Rational singularities and $q$-birational morphisms]
{Rational singularities and $q$-birational morphisms}

\begin{document}

\author{Donghyeon Kim}
\address{Department of mathematics, Yonsei University, 50 Yonsei-Ro, Seodaemun-Gu, Seoul 03722, Korea}
\email{whatisthat@yonsei.ac.kr}

\date{\today}

\begin{abstract}
In this paper, we generalize the notion of rational singularities for any reflexive sheaf of rank $1$, link our notion of rational singularities with the notion of rational singularities in \cite{Kovrational}, and prove generalizations of standard facts about rational singularities. Moreover, by using a definition of non-rational locus, we introduce the notion of $(B_{q+1})$ as a dual notion of well-known Serre's notion of $(S_{q+1})$, and prove a theorem about $q$-birational morphisms.
\end{abstract}

\maketitle

\section{Introduction}
In this paper, $q\ge 2$ is an integer, and $k$ is any algebraically closed field of arbitrary characteristic. Any variety is quasi-projective, separated, finite type and integral scheme over $k$. Any divisor is assumed to be Cartier unless otherwise stated.

Throughout the paper, we will assume that for any normal variety $X$, there is a proper birational morphism $f:X'\to X$ where $X'$ is smooth and $f$ is an isomorphism outside of the singular locus of $X$. We call such a map a \emph{resolution} of $X$. If $k$ is of characteristic $0$, our assumption follows from Hironaka's resolution of singularities, and if $k$ is of positive characteristic, only the threefold case is known to be affirmative (see \cite{Hir64} and \cite{Cut09}). Moreover, for a pair $(X,\Delta)$, a \emph{log resolution} $f:X'\to X$ of $(X,\Delta)$ is a resolution $f:X'\to X$ of $X$ such that $\Supp f^{-1}_*\Delta\cup \mathrm{Exc}(f)$ is simple normal crossing.

A variety $X$ is said to have \emph{rational singularities} if for any resolution $f:X'\to X$, $f_*\O_{X'}=\O_X, R^if_*\O_{X'}=0$ and $R^if_*\omega_{X'}=0$ for $i\ge 1$ (see \cite{Kollar2}, Definition 2.76, and introduction of \cite{rational}). The notion of rational singularities plays an important role in algebraic geometry. For instance, if $X$ is a variety with rational singularities and $f:X'\to X$ is a resolution of $X$, we have
$$ H^i(X',f^*\mathcal{E})\cong H^i(X,\mathcal{E})$$
for any $i\ge 0$ and any vector bundle $\mathcal{E}$ on $X$. In this paper, we generalize the notion to any reflexive sheaf of rank $1$.

The fundamental property of the notion of rational singularities is that if $X$ has rational singularities, then $X$ is CM. In this paper, from this perspective, we will define another analogous notion of rational singularities for a reflexive sheaf $\mathcal{F}$ of rank $1$ on a normal variety $X$.



\begin{defi}[See \cite{huybrechts2010geometry}, Definition 1.1.7]\thlabel{idefi2}
Let $X$ be any normal variety and $\mathcal{F}$ any coherent sheaf on $X$. Let us define the \emph{dual} of $\mathcal{F}$ by $\mathcal{F}^D:=\mathcal{H}om_{\O_X}(\mathcal{F},\omega_X)$, and the \emph{double dual} by $\mathcal{F}^{DD}:=((\mathcal{F})^D)^D$.
\end{defi}

Note that for a normal variety $X$ and a coherent sheaf $\mathcal{F}$ on $X$, $\mathcal{F}$ is reflexive if and only if the natural map $\mathcal{F}\to \mathcal{F}^{DD}$ is an isomorphism. See \thref{refe}.

\begin{definitionThm}\thlabel{idefi1}
Let $X$ be any normal variety and $\mathcal{F}$ any reflexive sheaf of rank $1$ on $X$. Suppose that $f:X'\to X$ is a generically finite morphism with $X'$ smooth. Then we have a natural map
$$ \theta_{\mathcal{F},f}:Rf_*(f^*\mathcal{F})^D\to R\mathcal{H}om_{\O_X}(\mathcal{F},\omega^{\bullet}_X)[-\dim X].$$
Moreover, if $f$ is a resolution, then the following are equivalent:
\begin{itemize}
    \item[(a)] $\theta_{\mathcal{F},f}$ is a quasi-isomorphism.
    \item[(b)] $R^if_*(f^*\mathcal{F})^{DD}=0$ for $i\ge 1$.
    \item[(c)] The natural map
    $$ H^i_x(X_x,\mathcal{F}_x)\to H^i_{f^{-1}(x)}(X'_x,(f^*\mathcal{F})^{DD}_x)$$
    is an isomorphism for each $i\ge 0$ and each point $x\in X$.
\end{itemize}
We say that $f$ is a \emph{weakly rational resolution} of $\mathcal{F}$ if one of the above is true for $f$, and $\mathcal{F}$ has \emph{weakly rational singularities} if there is a weakly rational resolution of $\mathcal{F}$.
\end{definitionThm}

\begin{defi}
Let $X$ be any normal variety and $\mathcal{F}$ any reflexive sheaf of rank $1$ on $X$.
\begin{itemize}
\item[(a)] If $\mathcal{F}$ has weakly rational singularities and is CM, then we say that $\mathcal{F}$ has \emph{rational singularities}.
\item[(b)] We say that a Weil divisor $D$ on $X$ has \emph{rational singularities} if $\O_X(D)$ has rational singularities.
\item[(c)] Let $f:X'\to X$ be a resolution. We say $f$ is \emph{$(KV_q)$ about $\mathcal{F}$} if the support of $R^if_*(f^*\mathcal{F})^D$ has codimension $\ge i+q+1$ for any $i\ge 1$. We call $\mathcal{F}$ $(KV_q)$ if there is a resolution of $X$ which is $(KV_q)$ about $\mathcal{F}$.
\end{itemize}
\end{defi}


Note that our definition bears a similarity to Definition 2.5 in \cite{Kovrational} and Definition 2.78 in \cite{Kollar2}. Let us recall the definition of rational singularities for a reduced pair $(X,D)$.

\begin{defi}[See Definition 2.5, and Theorem 2.9 in \cite{Kovrational}]
Let $X$ be a normal variety, and $D$ a reduced Weil divisor on $X$. We say that $(X,D)$ has \emph{rational singularities} if there is a log resolution $f:(X',D_{X'})\to (X,D)$ such that
\begin{itemize}
    \item[(a)] $(X,D)$ is a normal pair in the sense that the natural map $\O_X(-D)\to f_*\O_{X'}(-D_{X'})$ is an isomorphism, and
    \item[(b)] $R^if_*\O_{X'}(-D_{X'})=0$ for $i>0$,
\end{itemize}
where $D_{X'}$ is the strict transform of $D$ along $f$.
\end{defi}

Here, we highlight the distinctions between our definition and those in \cite{Kovrational} and \cite{Kollar2}:

\begin{itemize}
\item We used double dual to define the notion of rational singularities instead of strict transform. The advantage of using double dual is that the notion of normal pair in \cite{Kovrational} is unnecessary, and if we consider \thref{stronger}, our definition of rational singularities is minimal in sense of the requirement of the vanishing of higher direct image and CM-ness.

In particular, if a pair $(X,D)$ has rational singularities in the sense of \cite{Kovrational} and \cite{Kollar2}, then $-D$ has rational singularities in our sense. See \thref{pleaseseeme}.

\item It is obvious that some Cartier divisor on $X$ has rational singularities in our sense if and only if $X$ has rational singularities, and in contrast to our notion, it is far from obvious that for some reduced, effective and Cartier divisor $D$ on $X$, $(X,D)$ has rational singularities in the sense of \cite{Kovrational} if and only if $X$ has rational singularities. For the only if assertion, see Corollary 2.13 in \cite{Kovrational}. Note that the proof of Corollary 2.13 in \cite{Kovrational} implicitly used our notion of rational singularities.

\item We defined the notion of rational singularities for any reflexive sheaf of rank $1$, not only for a reduced Weil divisor.
\end{itemize}

One of the fundamental properties of klt and strongly $F$-regular varieties is that klt varieties over a characteristic $0$ field and strongly $F$-regular varieties over a positive characteristic field have rational singularities (see \cite{Elk81}, \cite{Kawamata} and \cite{HW19}). The same principle holds in our definition of rational singularities for reflexive sheaves of rank $1$.

\begin{thm}\thlabel{iex1}
Let $X$ be any normal variety over a characteristic $0$ field, $\Delta$ any effective $\Q$-Weil divisor such that $(X,\Delta)$ is dlt and $D$ any $\Q$-Cartier Weil divisor on $X$. Then $D$ has rational singularities.
\end{thm}

\begin{thm}\thlabel{iex11}
Let $X$ be any strongly $F$-regular variety over a positive characteristic field, and $D$ any Weil divisor on $X$ such that there is an integer $r$ which is relatively prime to $p$, and $rD$ is Cartier. Then $D$ has rational singularities.
\end{thm}

Note that \thref{iex1} is a generalization of Corollary 5.25 in \cite{KM} and the proof does not require the cyclic covering theory.	

One of the advantages of our notion is that our notion of rational singularities is well-behaved under finite morphisms under mild conditions, and it provides a partial answer to a question posed by Kollár (see Remark 2.81 (3) in \cite{Kollar2}).

\begin{thm}\thlabel{iquo}
Let $X,Y$ be normal varieties over a characteristic $0$ field, $p:Y\to X$ any finite {\'e}tale morphism and $\mathcal{F}$ any reflexive sheaf of rank $1$ on $X$. If $p^*\mathcal{F}$ has (resp. weakly) rational singularities, then $\mathcal{F}$ has (resp. weakly) rational singularities.
\end{thm}

\begin{thm}\thlabel{iquo2}
Let $X,Y$ be normal varieties over a characteristic $0$ field, $p:Y\to X$ any finite flat morphism and $\mathcal{F}$ any reflexive sheaf of rank $1$ on $X$. If $\mathcal{F}$ is $(KV_{\dim X})$ and $p^*\mathcal{F}$ has rational singularities, then $\mathcal{F}$ has rational singularities.
\end{thm}


For a normal variety $X$, one may want to measure how far a non-rational singularity is from being rational. For this purpose, the notion of non-rational locus is defined in \cite{AH} and \cite{Kovrational}. We also consider an analogous notion here.

\begin{defi}\thlabel{inon-RS}
Let $X$ be any normal variety, $\mathcal{F}$ any reflexive sheaf of rank $1$ on $X$, and $f:X'\to X$ a resolution. 
\begin{itemize}
    \item[(a)] We say that $f$ is \emph{$(KV_q)$ about $\mathcal{F}$} if the support of $R^if_*(f^*\mathcal{F})^D$ has codimension $\ge i+q+1$ for any $i\ge 1$. We call $\mathcal{F}$ $(KV_q)$ if one of the resolutions of $X$ is $(KV_q)$ about $\mathcal{F}$.
    \item[(b)] We say that $f$ is $(RS_q)$ about $\mathcal{F}$ if
\begin{itemize}
    \item[(i)] $f$ is $(KV_q)$ about $\mathcal{F}$, and
    \item[(ii)] for any point $x\in X$ with $\codim_X x\le q$, $\theta_{\mathcal{F},f}$ is a quasi-isomorphism after localizing $\theta_{\mathcal{F},f}$ at $x$.
\end{itemize}
We call $\mathcal{F}$ $(RS_q)$ if there is a resolution of $X$ which is $(RS_q)$ about $\mathcal{F}$.
\end{itemize}
\end{defi}

It is worth noting that $\mathcal{F}$ is $(KV_{\dim X})$ if and only if for some resolution $f:X'\to X$, $R^if_*(f^*\mathcal{F})^D=0$ for $i>0$.

\begin{rmk}
In \thref{inon-RS}, $\mathcal{F}$ is $(RS_{\dim X})$ if and only if $\mathcal{F}$ has rational singularities. Also in \thref{inon-RS}, condition (b) is equivalent to following: For any $i\ge 1$, $R^if_*(f^*\mathcal{F})^{DD}$ has the support of codimension $\ge q+1$ by almost the same argument in the proof of \thref{idefi1}.
\end{rmk}

Let us say that $\mathcal{F}$ is $(R_q)$ if there is an open subscheme $U\subseteq X$ of $X$ such that $\codim_X (X\setminus U)\ge q+1$, $U$ is smooth and $\mathcal{F}|_U$ is a vector bundle on $U$. The codimension of singular locus of $X$ determines the property of $(RS_q)$ of a reflexive sheaf $\mathcal{F}$ on $X$ of rank $1$.

\begin{lem}\thlabel{i??}
Let $X$ be any normal variety and $\mathcal{F}$ any reflexive sheaf of rank $1$ on $X$ with $(KV_q)$. If $X$ is $(R_q)$, then $\mathcal{F}$ is $(RS_q)$.
\end{lem}

\begin{ex}\thlabel{ex2}
We can give examples of reflexive sheaves of rank $1$ with $(RS_q)$. Let $X$ be any normal variety over $\C$ and suppose that $\Delta$ is an effective Weil divisor such that $(X,\Delta)$ is log canonical. If $D$ is a $\Q$-Cartier Weil divisor on $X$, then $D$ is $(RS_{q+1})$, where $q$ is the codimension of the union of the non-klt centers of $(X,\Delta)$.

It would be interesting to know whether any closed subvariety defined by the points $x\in X$ in which $(\theta_{\mathcal{F},f})_x$ is not a quasi-isomorphism for some resolution $f:X'\to X$ is a non-klt center of $(X,\Delta)$ or not. Note that if $D=0$, then there is an affirmative answer. See Theorem 1.2 in \cite{AH}.
\end{ex}


We say that $\mathcal{F}$ is $(S_{q+1})$ if $H^i_x(X,\mathcal{F})=0$ for $i<\min\{q+1,\codim_X x\}$ and any point $x\in X$. 



\begin{defi}
Let $X$ be any normal variety and $\mathcal{F}$ any reflexive sheaf of rank $1$ on $X$. We say that $\mathcal{F}$ is $(B_{q+1})$ if $H^i_x(X,\mathcal{F})=0$ for any $\dim X-q<i<\dim X$ and any closed point $x\in X$.
\end{defi}

For any normal variety $X$, $X$ is CM if and only if $\omega_X$ is CM. This equivalence is established in Lemma 0AWS of \cite{Stacks}. Nonetheless, when $q<\dim X-1$, the statement that $X$ is $(S_{q+1})$ if and only if $\omega_X$ is $(S_{q+1})$ does not hold. It appears to us that the more appropriate property for $\omega_X$ might be $(B_{q+1})$ rather than $(S_{q+1})$. For further discussion on this topic, one can refer to Section 4.3 in \cite{Kollar}.

\begin{thm}\thlabel{imain}
Let $X$ be any normal projective variety and $\mathcal{F}$ any reflexive sheaf of rank $1$ on $X$, which is $(RS_q)$. Then the following are equivalent:
\begin{itemize}
    \item[(a)] $\mathcal{F}$ is $(S_{q+1})$.
    \item[(b)] For any resolution $f:X'\to X$ with $f$ $(RS_q)$ about $\mathcal{F}$, $R^if_*(f^*\mathcal{F})^{DD}=0$ for any $1\le i<q$.
    \item[(c)] $\mathcal{F}^D$ is $(B_{q+1})$.
\end{itemize}
\end{thm}

Note that the proof of \thref{imain} is inspired by the proof of Lemma 3.3 in \cite{Kov}.

Using the notions of $(B_{q+1})$ and $(RS_q)$, we can prove the following theorem about the CM-ness of normal $\Q$-factorial varieties.

\begin{coro}\thlabel{imaincoro}
Let $X$ be any normal $\Q$-factorial variety with $(R_q)$. Assume that any Weil divisor on $X$ is $(KV_q)$, and $(S_{q+1})$. If $q\ge \ceil{\frac{\dim X+1}{2}}$, then any Weil divisor on $X$ is CM. In particular, $X$ is CM itself.
\end{coro}

\begin{rmk}
\thref{imaincoro} proves that in the setting of $\mathrm{char}\,k=0$, for any factorial variety $X$, if $X$ is $(RS_{\ceil{\frac{\dim X+1}{2}}})$ and $(S_{\ceil{\frac{\dim X+3}{2}}})$, then $X$ is CM. Indeed, any reflexive sheaf on $X$ of rank $1$ is a line bundle, and thus the Grauert-Riemenschneider vanishing theorem ensures $R^if_*(f^*\mathcal{F})^D=0$ for any reflexive sheaf $\mathcal{F}$ on $X$ of rank $1$, any resolution $f:X'\to X$, and any $i\ge 1$. Thus, \thref{imaincoro} gives us the claim.

One may think that Theorem 1.6 in \cite{HO74} is similar to \thref{imaincoro}. Hence, if we assume $X$ is $(R_q)$, then one can believe that there is a simple proof of \thref{imaincoro} using the local duality only.
\end{rmk}

We recall the notion of a \emph{$q$-birational morphism}, which was originally defined in \cite{trash}.

\begin{defi}\thlabel{ibirational}
Let $X,X'$ be any normal varieties over a characteristic $0$ field, and $f:X'\to X$ any proper birational morphism.
\begin{itemize}
\item[(a)] The \emph{center} of $f$ is the reduced closed subscheme $C$ of $X$ which is the image of the exceptional locus of $f$.
\item[(b)] We say $f$ is a \emph{$q$-birational morphism} if the exceptional locus has codimension $1$ and the center of $f$ has codimension $\ge q+1$.
\end{itemize}
\end{defi}

\begin{rmk}
For any $q$-birational morphism $f:X'\to X$ between normal varieties, if $X'$ is smooth, then $X$ is $(R_q)$.
\end{rmk}

The argument of the proof of \thref{imain} can be used to generalize Theorem 3.4 in \cite{trash} as follows:

\begin{thm}\thlabel{iyouaresecondapplication}
Let $X,X'$ be any normal varieties, $X'$ smooth and $f:X'\to X$ any $q$-birational morphism. Suppose that $D$ is any anti $f$-nef Cartier divisor on $X'$ such that $f_*D$ is $\Q$-Cartier and $(S_{q+1})$. Then $R^if_*\O_{X'}(D)=0$ for $1\le i<q$.
\end{thm}

The rest of the paper is organized as follows. We begin Section 2 by defining some basic notions and by stating and proving basic theorems. Section 3 is devoted to proving basic results about reflexive sheaves. In Section 4, we define (weakly) rational singularities and prove some basic theorems. In Section 5, we define the notion of $(B_{q+1})$ and prove the main theorem about that notion. Section 6 is devoted to defining $q$-birational morphisms and proving the main theorems about $q$-birational morphisms.

\begin{ack}
The author thanks his advisor Sung Rak Choi for his comments, questions, and discussions. He is grateful for his encouragement and support. The author is grateful to Fabio Bernasconi and S{\'a}ndor Kov{\'a}cs for their helpful comments on earlier versions of this paper. The author thanks the referee for careful reading, comments, and corrections of the previous version of this paper. We would like to acknowledge the assistance of ChatGPT in polishing the wording.
\end{ack}

\section{Preliminaries}
The section is devoted to collecting basic definitions and facts used in the paper.

\begin{notation}
Given two normal varieties $X,X'$, and a proper morphism $f:X'\to X$, and a (not necessarily closed) point $x\in X$, we use the following notation:

\begin{itemize}
    \item $X'_x:=X'\times_X \Spec \O_{X,x}$,
    \item $\dim x$ denotes the dimension of the closure of $x$,
    \item $\codim_X x:=\dim X-\dim x$.
\end{itemize}

For any coherent sheaves $\mathcal{F},\mathcal{F}'$ on $X,X'$ respectively and the inclusion $\jmath:X'_x\to X'$, we denote 

\begin{itemize}
    \item $\mathcal{F}'_{x}:=j^*\mathcal{F}'$,
    \item $\O_{X',x}:=(\O_{X'})_x$,
    $\omega_{X',x}:=(\omega_{X'})_x$.
\end{itemize}

Note that if $f:X'\to X$ is a resolution of $X$, then $X'_x$ is a regular scheme.

For simplicity, if $i\ge 0$, a point $x\in X$ and a coherent sheaf $\mathcal{F}$ on $X$ are given, set
$$H^i_x(X,\mathcal{F}):=H^i_x(X_x,\mathcal{F}_x).$$
Also, under the same conditions as above, if a resolution $f:X'\to X$ and a coherent sheaf $\mathcal{F}'$ on $X'$ are given, let
$$ H^i_{f^{-1}(x)}(X',\mathcal{F}'):=H^i_{f^{-1}(x)}(X'_x,\mathcal{F}'_x).$$
\end{notation}

\begin{defi}
Let $X$ be any normal variety and $\mathcal{F}$ any torsion-free sheaf on $X$.
\begin{itemize}
    \item[(a)] We say that $\mathcal{F}$ is $(R_q)$ if there is an open subscheme $U\subseteq X$ such that $\codim_X (X\setminus U)\ge q+1$, $U$ is smooth and $\mathcal{F}|_U$ is a vector bundle on $U$.
    \item[(b)] We say that $X$ is $(R_q)$ if $\O_X$ is $(R_q)$.
    \item[(c)] We say that $\mathcal{F}$ is $(S_{q+1})$ if $H^i_x(X,\mathcal{F})=0$ for $i<\min\{q+1,\codim_X x\}$ and any point $x\in X$.
    \item[(d)] We say that $\mathcal{F}$ is \emph{Cohen-Macaulay (CM)} if $\mathcal{F}$ is $(S_{\dim X})$.
    \item[(e)] We say that $X$ is \emph{Cohen-Macaulay (CM)} if $\O_X$ is CM.
    \item[(f)] We say that $\mathcal{F}$ is \emph{reflexive} if $\mathcal{F}$ is $(S_2)$.
\end{itemize}
\end{defi}

\begin{rmk}
If $X$ is any normal variety with $(R_q)$, then any reflexive sheaf of rank $1$ on $X$ is also $(R_q)$. An analogue for $(S_{q+1})$ does not hold. Indeed, let $X:=\Spec k[x,y,z,w]/(xy-zw)$. Then $X$ is CM but not any reflexive sheaf of rank $1$ on $X$ is CM. See 3.15 in \cite{Kollar2}.
\end{rmk}

We will use derived category machinery in the paper, especially in Section 4 and Section 5. Hence, it is worth stating the \emph{local duality} and the \emph{Grothendieck duality}. For a separated scheme $X$ over $k$, let us define the notion of \emph{normalized dualizing complex} $\omega^{\bullet}_X$ of $X$ by $f^!\omega_{\Spec k}$, where $f:X\to \Spec k$ is the structure morphism.

\begin{thm}[See Lemma 0A85 in \cite{Stacks}]\thlabel{localduality}
Let $X$ be any variety, $x\in X$ any closed point and $E$ the injective hull of the residue field of $\O_{X,x}$. Write $\m_{X,x}$ the maximal ideal of $\O_{X,x}$, and $Z=V$. Then
$$ R\Hom_{\O_{X,x}}(K,\omega^{\bullet}_X)^{\wedge}_x\cong R\Hom_{\O_{X,x}}(R\Gamma_Z(K),E[0])$$
for any $i$ and $K\in D(X)$, where $(-)^{\wedge}_x$ denotes the derived completion along $\m_{X,x}$, and $D(X)$ denotes the derived category of bounded complexes of coherent sheaves on $X$.
\end{thm}

\begin{thm}[See 0AU3, (4) in \cite{Stacks}]
Let $X,X'$ be any noetherian separated schemes over $k$, $f:X'\to X$ any proper morphism of varieties. For any $K\in D(X')$,
$$ R\mathcal{H}om_{\O_X}(Rf_*K,\omega^{\bullet}_{X}) \cong Rf_*R\mathcal{H}om_{\O_{X'}}(K,\omega^{\bullet}_{X'})\text{ in }D(X).$$
\end{thm}

Let us prove the following duality result. For a similar lemma, see Proposition 11.6 in \cite{singularitiesofpairs}.

\begin{lem}\thlabel{warmup1}
Let $f:X'\to X$ be any proper birational morphism between normal varieties and $\mathcal{E}$ any vector bundle on $X'$. Fix any point $x\in X$. Suppose that $X'$ is Gorenstein. Then for $i\ge 0$, $(R^if_*\mathcal{E})_x=0$ if and only if $H^{\codim_X x-i}_{f^{-1}(x)}(X',\omega_{X'}\otimes \mathcal{E}^{\vee})=0$.
\end{lem}

\begin{proof}
Let $k$ be the residue field of $\O_{X,x}$ and $E$ the injective hull of $k$ as an $\O_{X,x}$-module. We have
$$ 
\begin{aligned}
R\Gamma_{f^{-1}(x)}(X',\omega_{X'}\otimes \mathcal{E}^{\vee})&\cong R\Gamma_x(X,Rf_*(\omega_{X'}\otimes \mathcal{E}^{\vee}))
\\ &\cong \Hom_{\O_{X,x}}(R\mathcal{H}om_{\O_X}(Rf_*(\omega_{X'}\otimes \mathcal{E}^{\vee})_x,\omega^{\bullet}_{X,x}),E)
\\ &\cong \Hom_{\O_{X,x}}(Rf_*R\mathcal{H}om_{\O_{X'}}(\omega_{X'}\otimes \mathcal{E}^{\vee},\omega_{X'})_x[\codim_X x],E)
\\ &\cong \Hom_{\O_{X,x}}(Rf_*\mathcal{H}om_{\O_{X'}}(\omega_{X'}\otimes \mathcal{E}^{\vee},\omega_{X'})_x[\codim_X x],E)
\\ &\cong \Hom_{\O_{X,x}}((Rf_*\mathcal{E})_x[\codim_X x],E),
\end{aligned}
$$
where we used the Leray spectral sequence in the first equality, the local duality in the second equality, the Grothendieck duality in the third equality, the fact that $\omega_{X'}$ is a line bundle on $X'$ in the forth equality. Hence, since $E$ is an injective $\O_{X,x}$-module, we have the assertion.
\end{proof}

Let us prove the following corollary.

\begin{coro}[See Lemma 3.5.10 in \cite{Nakayama} and Corollary 11.7 in \cite{singularitiesofpairs}]\thlabel{warmup2}
Let $X$ be any normal projective variety over a field of characteristic $0$, $f:X'\to X$ any resolution of $X$ and $D$ any $\Q$-Cartier Weil divisor on $X'$, which is anti $f$-nef. Then $f_*\O_{X'}(D)=\O_X(f_*D)$.
\end{coro}

\begin{proof}
Note that both $\O_X(f_*D)$ and $(f_*\O_{X'}(D))^{\vee\vee}$ are reflexive, and they are isomorphic in codimension $1$. Hence, $\O_X(f_*D)\cong (f_*\O_{X'}(D))^{\vee\vee}$ by Theorem 1.12 in \cite{HarGor}, and thus it suffices to show that $f_*\O_{X'}(D)$ is reflexive. Let $x\in X$ be any point of codimension $\ge 2$.

Consider the following spectral sequence
$$ E^{st}_2=H^s_x(X,R^tf_*\O_{X'}(D))\implies H^{s+t}_{f^{-1}(x)}(X',\O_{X'}(D)).$$
Inspecting the spectral sequence, we have $E^{10}_2=E^{10}_{\infty}$. Hence, by considering the edge map $E^{10}_{\infty}\to E^1$, it suffices to show that $H^1_{f^{-1}(x)}(X',\O_{X'}(D))=0$. Note that by the relative Kawamata-Viehweg vanishing theorem (Theorem 1-2-3 in \cite{Kawamata}), $R^if_*\O_X(K_{X'}-D)=0$ for $i\ge 1$. Thus, by \thref{warmup1}, we have the assertion.
\end{proof}

\begin{rmk}
Note that it is a special case of \cite{Nakayama}, Lemma 3.5.10. Nakayama proved the lemma using the relative Zariski decomposition.
\end{rmk}

We will use the following lemma without any mention.

\begin{lem}[See Lemma 2.2 in \cite{+}]
For any normal variety $X$ and any coherent sheaf $\mathcal{F}$ on $X$, we have
$$ \mathcal{H}om_{\O_X}(\mathcal{F},\omega^{\bullet}_X)[-\dim X]=\mathcal{H}om_{\O_X}(\mathcal{F},\omega_X).$$
\end{lem}

\begin{proof}
We may consider the following exact triangle
$$ \omega_X[-\dim X]\to \omega^{\bullet}_X\to C\to $$
for some complex $C$ in $D(X)$. If we apply $R\mathcal{H}om_{\O_X}(\mathcal{F},-)$ to this triangle, then we obtain the following exact triangle
$$ R\mathcal{H}om_{\O_X}(\mathcal{F},\omega_X[-\dim X])\to R\mathcal{H}om_{\O_X}(\mathcal{F},\omega^{\bullet}_X)\to R\mathcal{H}om_{\O_X}(\mathcal{F},C)\to .$$
Note that $C$ has cohomological degree $\ge 1$ and hence $\mathcal{H}om_{\O_X}(\mathcal{F},C)=0$. Thus, the long exact cohomology sequence gives us the assertion.
\end{proof}

\section{Reflexive sheaves}
In this section, we collect important facts about reflexive sheaves. For a variety $X$, the generic point $\eta\in X$ and a coherent sheaf $\mathcal{F}$ on $X$, $\mathrm{rank}\,\mathcal{F}$ denotes the dimension of $\mathcal{F}_{\eta}$ over the function field of $X$.


\begin{rmk}
For any normal Gorenstein variety $X$ and any reflexive sheaf $\mathcal{F}$ of rank $1$ on $X$, we have $\mathcal{F}^D=\mathcal{F}^{\vee}\otimes \omega_X$.

Moreover, we can construct a natural map $\mathcal{F}\to \mathcal{F}^{DD}$. Indeed, for any local section $a\in \mathcal{F}$, consider a map
$$ \varphi:\mathcal{F}\to \mathcal{F}^{DD},a\mapsto (f\mapsto f(a))\,\text{ for }f\in \mathcal{F}^D:=\mathcal{H}om_{\O_X}(\mathcal{F},\omega_X).$$
Note that if $X$ is normal, the map $\mathcal{F}^D\to \mathcal{F}^{DDD}$ is an isomorphism, because $\mathcal{F}^{D}$ is a reflexive sheaf. See Lemma 0AY4 in \cite{Stacks}, Lemma 0AWE in \cite{Stacks}, and \thref{refe} below.
\end{rmk}

Now, let us prove the following two lemmas.

\begin{lem}
Let $X$ be any normal variety and $\mathcal{F}$ any coherent sheaf on $X$. Then the following are equivalent:
\begin{itemize}
    \item[(a)] $\mathcal{F}$ is torsion-free.
    \item[(b)] The map $\mathcal{F}\to \mathcal{F}^{DD}$ is injective.
\end{itemize}
\end{lem}

\begin{proof}
It is worth noting that for any variety $X$, $\omega_X$ is $(S_2)$. See Lemma 0AWE in \cite{Stacks}. Hence, if $X$ is normal, then $\omega_X$ is reflexive.

If $\mathcal{F}$ is torsion-free, then the kernel $\mathcal{K}$ of the map is torsion-free because any subsheaf of a torsion-free sheaf is also torsion-free. Furthermore, since $\omega_X$ is reflexive, $\mathcal{F}^{DD}$ is torsion-free by Lemma 0AY4 in \cite{Stacks}. Consider the fact that for every point $x\in X$ of codimension $1$, $\mathcal{F}_x\to \mathcal{F}^{DD}_x$ is an isomorphism (which follows from Lemma 0CC4 in \cite{Stacks}, and the fact that $\O_{X,x}$ is a discrete valuation ring). Then $\mathcal{K}=0$ because $\mathrm{rank}\,\mathcal{K}=0$.

For the converse, $\mathcal{F}^{DD}$ is torsion-free. By our assumption, $\mathcal{F}$ is torsion-free because any subsheaf of a torsion-free sheaf is also torsion-free.
\end{proof}

\begin{lem}\thlabel{refe}
Let $X$ be any normal variety and $\mathcal{F}$ any coherent sheaf on $X$. Then the following are equivalent:
\begin{itemize}
    \item[(a)] $\mathcal{F}$ is reflexive.
    \item[(b)] The natural map $\mathcal{F}\to \mathcal{F}^{DD}$ is an isomorphism.
\end{itemize}
\end{lem}

\begin{proof}
Let us assume that $\mathcal{F}$ is reflexive. By Lemma 0AY4 in \cite{Stacks}, we know that $\mathcal{F}^{DD}$ is also reflexive and $\mathcal{F}\to \mathcal{F}^{DD}$ is an isomorphism outside a codimension $\ge 2$ closed subscheme of $X$. Hence the map is an isomorphism by \cite{HarGor}, Proposition 1.11.

For the converse, let $\mathcal{E}_1\to \mathcal{E}_0\to \mathcal{F}^D\to 0$ be any resolution of $\mathcal{F}$ in which $\mathcal{E}_0,\mathcal{E}_1$ are vector bundles. By taking $\mathcal{H}om_{\O_X}(-,\omega_X)$, we have the following exact sequence
$$ 0\to \mathcal{F}\cong \mathcal{F}^{DD}\to \mathcal{E}^D_0\to \mathcal{E}^D_1,$$
and $\mathcal{E}^D_i$ are reflexive for $i=0,1$. Hence, $\mathcal{F}$ is reflexive (see Lemma 0EBG in \cite{Stacks}).
\end{proof}

\begin{rmk}
If $X$ is not normal, then such an equivalence is false. See Remark 0AY1 in \cite{Stacks}.

Let us note an important fact about reflexive sheaves. For any normal variety $X$ and any reflexive sheaf $\mathcal{F}$ on $X$, $\mathcal{F}$ is a vector bundle outside codimension $\ge 2$ closed subscheme of $X$. See Lemma 0AY6 in \cite{Stacks}.
\end{rmk}

We will use the following three lemmas without any mention. We believe that they are well-known to experts but we cannot find any reference about them. Hence, we include their proof.

\begin{lem}
Let $X$ be any variety and $\varphi:\mathcal{F}\to \mathcal{G}$ any map between torsion-free sheaves on $X$. Suppose that $\varphi$ is generically injective. Then $\varphi$ is an injection.
\end{lem}

\begin{proof}
Let $\mathcal{K}$ be the kernel of $\varphi$. Note that $\mathcal{K}$ is a subsheaf of a torsion-free sheaf $\mathcal{F}$. By the condition on $\varphi$, $\mathrm{rank}\,\mathcal{K}=0$ and hence $\mathcal{K}=0$, because any subsheaf of torsion-free sheaf is also torsion-free.
\end{proof}

\begin{lem}\thlabel{goo}
Let $X,X'$ be any normal varieties and $f:X'\to X$ any proper birational morphism. Suppose that $\mathcal{F}$ is any reflexive sheaf on $X$. Then $\mathcal{F}\cong f_*(f^*\mathcal{F})^{DD}$.
\end{lem}

\begin{proof}
Let us consider the adjunction property
$$ \mathcal{H}om_{\O_X}(\mathcal{F},f_*(f^*\mathcal{F})^{DD})\cong f_*\mathcal{H}om_{\O_{X'}}(f^*\mathcal{F},(f^*\mathcal{F})^{DD})$$
and the map $f^*\mathcal{F}\to (f^*\mathcal{F})^{DD}$. Then we can construct the following map $\mathcal{F}\to f_*(f^*\mathcal{F})^{DD}$. Note that the map is an injection because it is an isomorphism on the regular locus of $X$ and $\mathcal{F},f_*(f^*\mathcal{F})^{DD}$ are torsion-free sheaves.

Let $\mathcal{Q}$ be the cokernel. Then we have the following exact sequence
$$ 0\to \mathcal{F}\to f_*(f^*\mathcal{F})^{DD}\to \mathcal{Q}\to 0.$$
Note that the support of $\mathcal{Q}$ has codimension $\ge 2$. Let $x\in X$ be any point of codimension $\ge 2$. Then the local cohomology exact sequence tells us that $H^0_x(X,\mathcal{Q})=0$. Hence $x$ is not an associated point of $\mathcal{Q}$ (cf. Lemma 0EEZ in \cite{Stacks}). Thus $\mathcal{Q}=0$ by Lemma 05AG in \cite{Stacks}.
\end{proof}

\begin{lem}
Let $X,Y$ be any normal varieties and $p:X'\to X$ any proper flat morphism. Suppose that $\mathcal{F}$ is any coherent sheaf on $X$. Then $(p^*\mathcal{F})^{DD}=p^*(\mathcal{F}^{DD}).$
\end{lem}

\begin{proof}
Note that the pullback of a reflexive sheaf along a flat morphism is also reflexive. See Proposition 1.8 in \cite{Har80}.

Taking the natural map $\mathcal{F}\to \mathcal{F}^{DD}$, we have a map $\varphi:(p^*\mathcal{F})^{DD}\to p^*(\mathcal{F}^{DD})$. We know that both $(p^*\mathcal{F})^{DD},p^*(\mathcal{F}^{DD})$ on $Y$ are reflexive and $\varphi$ is an isomorphism outside codimension $\ge 2$ closed subscheme of $Y$ (see \thref{refe}). Hence by \cite{HarGor}, Proposition 1.11, we have the assertion.
\end{proof}

\section{Rational singularities}
Throughout the remainder of this paper, for simplicity, assume that any reflexive sheaf has rank $1$, unless explicitly mentioned otherwise. Note that the condition of being rank $1$ is only used in the proof of \thref{idefi1} to ensure that 
$$\mathcal{E}xt^i_{\O_{X'}}((f^*\mathcal{F})^{DD},\omega_{X'})=0\text{ for }i>0$$
if we use the notations in \thref{idefi1} below, and hence we believe such condition can be loosened easily.


\begin{proof}[Proof of \thref{idefi1}]
The following proof is inspired by the argument in the proof of Lemma 3.7 in \cite{SchTa}.

Let us construct $\theta_{\mathcal{F},f}$. Consider a composition $\mathcal{F}\to f_*(f^*\mathcal{F})^{DD}\to Rf_*(f^*\mathcal{F})^{DD}$. We may take $R\mathcal{H}om_{\O_X}(-,\omega^{\bullet}_X)$ on the map and it gives
$$ R\mathcal{H}om_{\O_X}(Rf_*(f^*\mathcal{F})^{DD},\omega^{\bullet}_X)\to R\mathcal{H}om_{\O_X}(\mathcal{F},\omega^{\bullet}_X).$$

Now, we may compute the left-hand side by the Grothendieck duality as follows:
$$
\begin{aligned}
R\mathcal{H}om_{\O_X}(Rf_*(f^*\mathcal{F})^{DD},\omega^{\bullet}_X)&=Rf_*R\mathcal{H}om_{\O_{X'}}((f^*\mathcal{F})^{DD},\omega_{X'})[\dim X]
\\ &=Rf_*\mathcal{H}om_{\O_{X'}}((f^*\mathcal{F})^{DD},\omega_{X'})[\dim X]
\\ &=Rf_*(f^*\mathcal{F})^{D}[\dim X],
\end{aligned}
$$
where we used the fact that $(f^*\mathcal{F})^{DD}$ is a line bundle on $X'$ on the second equality. Therefore, we have a map
$$ \theta_{\mathcal{F},f}:Rf_*(f^*\mathcal{F})^D\to R\mathcal{H}om_{\O_X}(\mathcal{F},\omega^{\bullet}_X)[-\dim X].$$

For $(a)\iff (b)$, if $(b)$ is true, then $\mathcal{F}\cong Rf_*(f^*\mathcal{F})^{DD}$. Thus, $\theta_{\mathcal{F},f}$ is a quasi-isomorphism and $(a)$ is true. For the converse, we may take $R\mathcal{H}om_{\O_X}(-,\omega^{\bullet}_X)$ on $\theta_{\mathcal{F},f}$. See 0AU3 (3) in \cite{Stacks}.

For $(a)\iff (c)$, we may take the local duality to $\theta_{\mathcal{F},f}$. Note that the dual of $\theta_{\mathcal{F},f}$ is
$$ R\Gamma_x(X_x,\mathcal{F}_x)\to R\Gamma_x(X_x,Rf_*(f^*\mathcal{F})^{DD})$$
and the right hand side is $R\Gamma_{f^{-1}(\{x\})}(X'_{x},(f^*\mathcal{F})^{DD})$ by the Leray spectral sequence.
\end{proof}


\begin{rmk}\thlabel{comp}
If we use the notions of \thref{idefi1}, then $\mathcal{H}^0(\theta_{\mathcal{F},f})$ is
$$\mathcal{H}^0(\theta_{\mathcal{F},f}):f_*(f^*\mathcal{F})^D\to \mathcal{F}^D=\mathcal{H}om_{\O_X}(\mathcal{F},\omega_X).$$
If we let $\mathcal{F}=\O_X$, then $\mathcal{H}^0(\theta_{\mathcal{F},f})$ is the trace map introduced in \cite{KM}, Proposition 5.77.

Let $f:X'\to X,g:X''\to X'$ be generically finite morphisms between normal varieties, and $\mathcal{F}$ any coherent sheaf on $X$. If $X',X''$ are smooth, then the map 
$$\mathcal{H}^0(\theta_{\mathcal{F},f\circ g}):(f\circ g)_*((f\circ g)^*\mathcal{F})^D\to \mathcal{F}^D$$
can be factored as follows:
$$\begin{tikzcd}
(f\circ g)_*((f\circ g)^*\mathcal{F})^D \ar["\mathcal{H}^0(\theta_{(f^*\mathcal{F})^{DD},g})"]{rr}& & f_*(f^*\mathcal{F})^D \ar["\mathcal{H}^0(\theta_{\mathcal{F},f})"]{rr}& & \mathcal{F}^D.
\end{tikzcd}
$$
Indeed, let $u:(f^*\mathcal{F})^{DD}\to Rg_*((f\circ g)^*\mathcal{F})^{DD}$ be the natural map. Then by definition,
$$\theta_{g,(f^*\mathcal{F})^{DD}}=R\mathcal{H}om_{\O_{X'}}(u,\omega^{\bullet}_{X'})[-\dim X].$$ If we consider
$$ \mathcal{F}\to Rf_*(f^*\mathcal{F})^{DD}\overset{Rf_*(u)}{\to} R(f\circ g)_*((f\circ g)^*\mathcal{F})^{DD},$$
by taking $R\mathcal{H}om_{\O_X}(-,\omega^{\bullet}_X)$ on that, then we have the maps
\begin{equation}\label{gr}
R(f\circ g)_*((f\circ g)^*\mathcal{F})^D\overset{\alpha}{\to} Rf_*(f^*\mathcal{F})^D\to \mathcal{H}om_{\O_X}(\mathcal{F},\omega^{\bullet}_X)[-\dim X].
\end{equation}
If we use the Grothendieck duality, then we obtain
$$\alpha=Rf_*\circ R\mathcal{H}om_{\O_{X'}}(u,\omega^{\bullet}_{X'})[-\dim X]=Rf_*\circ \theta_{g,(f^*\mathcal{F})^{DD}}.$$
If we apply $\mathcal{H}^0$ on \eqref{gr}, then we have the assertion.
\end{rmk}

\begin{prop}\thlabel{tr}
Let $X$ be any normal variety and $\mathcal{F}$ any reflexive sheaf of rank $1$ on $X$ which has weakly rational singularities. Then the following are equivalent:
\begin{itemize}
    \item[(a)] $\mathcal{F}$ is $(S_{q+1})$.
    \item[(b)] For some weakly rational resolution $f:X'\to X$, the support of $R^if_*(f^*\mathcal{F})^D$ has codimension $\ge i+q+1$ for any $i\ge 1$.
    \item[(c)] The support of $R^if_*(f^*\mathcal{F})^D$ has codimension $\ge i+q+1$ for any $i\ge 1$ and any weakly rational resolution $f:X'\to X$ of $X$.
\end{itemize}
\end{prop}

\begin{proof}
For $(a)\implies (c)$, given any weakly rational resolution $f:X'\to X$, by our assumption, there is an isomorphism 
$$\theta_{\mathcal{F},f}:Rf_*(f^*\mathcal{F})^D\cong R\mathcal{H}om_{\O_X}(\mathcal{F},\omega^{\bullet}_X)[-\dim X].$$
Moreover, for any point $x\in X$ with $\codim_X x\le i+q$, $H^{\codim_X x-i}_x(X,\mathcal{F})=0$ holds, because $\mathcal{F}$ is $(S_{q+1})$. Hence, by the local duality, $\Ext^{-\codim_X x+i}_{\O_{X,x}}(\mathcal{F}_x,\omega^{\bullet}_{X,x})=0$ for such $i$. Thus, $R^if_*(f^*\mathcal{F})^D_x=0$ and the support of $R^if_*(f^*\mathcal{F})^D$ does not contain $x$. 

$(c)\implies (b)$ is trivial.

For $(b)\implies (a)$, consider a point $x\in X$. If $\codim_X x\le q$, then $R^if_*(f^*\mathcal{F})^D_x=0$ and $\Ext^{-\codim_X x+i}_{\O_{X,x}}(\mathcal{F}_x,\omega^{\bullet}_{X,x})=0$ hold for any $i\ge 1$. Therefore, the local duality gives 
$$H^{\codim_X x-i}_x(X,\mathcal{F})=0\text{ for any }i\ge 1.$$

If $\codim_X x\ge q+1$ and $R^if_*(f^*\mathcal{F})^D_x=0$, by the local duality, $H^{\codim_X x-i}_x(X,\mathcal{F})=0$. Furthermore, if $\codim_X x-i<q$, then $R^if_*(f^*\mathcal{F})^D_x=0$. Hence, for $0\le j<q$, $H^j_x(X,\mathcal{F})=0$. 
\end{proof}



Now, we can define the notion of \emph{rational singularities} for any reflexive sheaf of rank $1$.

\begin{ex}\thlabel{ex}
Having weakly rational singularities does not ensure having rational singularities.

For example, take $X:=\mathrm{Spec}\;k[x,y,z,w]/(xy-zw)$, $A:=\{x=z=0\}$, the blowup $f:X'\to X$ along $A$ and the exceptional divisor $E$ on $X'$. We know that $f$ is a small resolution and $(f^*\O_X(-mA))^{DD}=\O_{X'}(-mE)$. Then for $m\gg 0$, $R^if_*(f^*\O_{X}(-mA))^{DD}=0$ for $i\ge 1$. Thus, $-mA$ has weakly rational singularities for such $m$. However, for $m\ge 2$, $-mA$ is not CM as in (3.15) in \cite{Kollar2} and hence $-mA$ does not have rational singularities.
\end{ex}

Let us prove the following lemma. Similar statements are proved in Lemma 3.2 in \cite{Alex} and Theorem 7.1.1 in \cite{Fujino}.

\begin{lem}\thlabel{twospectral}
Let $f:X'\to X$ be any proper birational morphism of varieties and $\mathcal{F},\mathcal{F}'$ any coherent sheaves on $X'$. For any given positive integer $n\ge 1$, suppose that the three conditions hold:
\begin{itemize}
\item[(1)] There is an injection $\imath:\mathcal{F}\to \mathcal{F'}$,
\item[(2)] $\imath$ induces an isomorphism $f_*\mathcal{F}\cong f_*\mathcal{F}'$, and
\item[(3)] $R^if_*\mathcal{F}'=0$ for $1\le i<n$.
\end{itemize}
Then we have $R^if_*\mathcal{F}=0$ for $1\le i<n$.
\end{lem}

\begin{proof}
Consider any point $x\in X$. We have the following two spectral sequences
$$ \begin{aligned}&{}^1E^{st}_2=H^s_x(X,R^tf_*\mathcal{F})\implies H^{i+j}_{f^{-1}(x)}(X',\mathcal{F}),
 	\\ &{}^2E^{st}_2=H^s_x(X,R^tf_*\mathcal{F}')\implies H^{i+j}_{f^{-1}(x)}(X',\mathcal{F'}).
 \end{aligned}
$$
From (3), we deduce ${}^2E^{i0}_2={}^2E^{i0}_{\infty}$, making the edge map ${}^2E^{i0}_2\to {}^2E^i$ injective for all $0\le i\le n$. Using (1), we have the following diagram:
$$
\begin{tikzcd}
{}^1E^{i}=H^i_{f^{-1}(x)}(X',\mathcal{F}) \ar{r}& {}^2E^i=H^i_{f^{-1}(x)}(X',\mathcal{F}') \\
{}^1E^{i0}_2=H^i_x(X,f_*\mathcal{F}) \ar["\gamma_i"]{r}\ar["\alpha_i"]{u}& {}^2E^{i0}_2=H^i_x(X,f_*\mathcal{F}').\ar["\beta_i"]{u}
\end{tikzcd}
$$
Given (2), for any $i$, $\gamma_i$ is an isomorphism. Since $\beta_i$ is an injection for $1\le i\le n$, $\alpha_i$ is also an injection for such $i$.

Now, we may use induction. Indeed, suppose that for a positive integer $1<n'\le n$, $R^if_*\mathcal{F}=0$ for $1\le i<n'$. Our goal is to show $R^{n'}f_*\mathcal{F}=0$. By the induction hypothesis, we have ${}^1E^{0n'}_2={}^1E^{0n'}_n$ and ${}^1E^{(n'+1)0}_2={}^1E^{(n'+1)0}_{n}$. Hence, there is an exact sequence
$$ 0\to {}^1E^{0n'}_2\to {}^1E^{(n'+1)0}_2\overset{\beta_{n'+1}}{\to} {}^1E^{n'+1}.$$
Since $\beta_{n'+1}$ is an isomorphism, we obtain $H^0_x(X,R^{n'+1}f_*\mathcal{F})={}^1E^{(n'+1)0}_2=0$ and hence $R^{n'+1}f_*\mathcal{F}$ does not have (the arbitrary point) $x\in X$ as an associated point. Thus $R^{n'+1}f_*\mathcal{F}=0$ and we have the assertion.
\end{proof}

Now, we have an alternative description of the notion of rational singularities.

\begin{prop}\thlabel{stronger}
Let $X$ be any normal variety and $\mathcal{F}$ any reflexive sheaf of rank $1$ on $X$. Then $\mathcal{F}$ has rational singularities if and only if for a resolution $f:X'\to X$, there is a line bundle $\mathcal{L}_{\mathcal{F},f}$ on $X'$ such that the three following conditions hold:
\begin{itemize}
\item[(i)] $f_*\mathcal{L}_{\mathcal{F},f}=\mathcal{F}$,
\item[(ii)] $R^if_*\mathcal{L}_{\mathcal{F},f}=0$ for $i\ge 1$, and	
\item[(iii)] $R^if_*\mathcal{L}^D_{\mathcal{F},f}=0$ for $i\ge 1$.
\end{itemize} 
\end{prop}

\begin{proof}
For the only if direction, set $\mathcal{L}_{\mathcal{F},f}:=(f^*\mathcal{F})^{DD}$. Now, apply \thref{goo} and \thref{idefi1}.

For the if direction, let us consider the counit map $f^*f_*\mathcal{L}_{\mathcal{F},f}\to \mathcal{L}_{\mathcal{F},f}$. By taking double dual, we obtain an injection $(f^*f_*\mathcal{L}_{\mathcal{F},f})^{DD}\to \mathcal{L}_{\mathcal{F},f}.$ Since (i) holds, $(f^*f_*\mathcal{L}_{\mathcal{F},f})^{DD}=(f^*\mathcal{F})^{DD}$. Thus, there is an injection $(f^*\mathcal{F})^{DD}\to \mathcal{L}_{\mathcal{F},f}$.

Considering \thref{goo,} \thref{twospectral}, (i) and (ii), we have $R^if_*(f^*\mathcal{F})^{DD}=0$ for $i\ge 1$. Let us use the notations in \thref{twospectral}. By (iii) and \thref{warmup1}, we obtain
$$ H^i_{f^{-1}(x)}(X',\mathcal{L}_{\mathcal{F},f})=0\text{ for }i<\codim_X x.$$
Since $\beta_i$ is an isomorphism for $i\ge 1$, we deduce
$$H^i_x(X,\mathcal{F})=H^i_{f^{-1}(x)}(X',\mathcal{L}_{\mathcal{F},f})=0\text{ for }i<\codim_X x$$
and hence $\mathcal{F}$ is CM. Thus $\mathcal{F}$ has rational singularities.
\end{proof}

Now, we can see that our notion of rational singularities is weaker than the notion of rational singularities in \cite{Kovrational} and \cite{Kollar2}.

\begin{coro}\thlabel{pleaseseeme}
Let $X$ be a normal variety over a characteristic $0$ field, and $D$ a reduced Weil divisor on $X$. If $(X,D)$ has rational singularities in the sense of \cite{Kovrational}, then $-D$ has rational singularities.
\end{coro}

\begin{proof}
Let $f:(X',D_{X'})\to (X,D)$ be a log resolution of $(X,D)$. Then we may put $\mathcal{L}_{\O_X(-D),f}:=\O_{X'}(-D_{X'})$ and use \thref{stronger}. Now, use Theorem 2.9 in \cite{Kovrational}.
\end{proof}

Before proving that any $\Q$-Cartier Weil divisor on a dlt pair $(X,\Delta)$ over a characteristic $0$ field has rational singularities (see \thref{iex1}), let us prove the following lemma.

\begin{lem}\thlabel{areyouwrong}
Let $X$ be a normal variety, $f:X'\to X$ a resolution and $D$ a $\Q$-Cartier Weil divisor on $X$. For any Cartier divisor $L$ on $X'$ such that $\O_{X'}(L)=(f^*\O_X(D))^{DD}$, there is an effective $f$-exceptional divisor $E$ such that $L+E\sim_{\Q}f^*D$.
\end{lem}

\begin{proof}
Let $d$ be a positive integer such that $dD$ is Cartier. Let us consider the following map
$$ \O_X(D)\otimes \cdots \otimes \O_X(D)\to \O_X(dD).$$
Then by applying $f^*$ to both sides, and reflexing, we have an injection
$$ \O_{X'}(dL)=(f^*\O_X(D)\otimes \cdots \otimes f^*\O_X(D))^{DD}\to \O_{X'}(f^*(dD)).$$
Thus $f^*(dD)-dL$ is effective. We may choose $L$ so that $f_*L=D$ holds. Moreover, since $f_*(f^*(dD)-dL)=dD-dD=0$ holds, $f^*(dD)-dL$ is $f$-exceptional. Now, put $E:=\frac{1}{d}(f^*(dD)-dL)$.
\end{proof}

We now prove \thref{iex1} and \thref{iex11}.


\begin{proof}[Proof of \thref{iex1}]
It suffices to prove the case when $(X,\Delta)$ klt. In fact, by Proposition 2.43 in \cite{KM}, for any ample Cartier divisor $H$ on $X$, there is $0<\varepsilon\ll 1$, a rational number $c>0$, and an effective $\Q$-divisor $D\equiv \Delta+cH$ such that $(X,(1-\varepsilon)\Delta+D)$ is klt. Now, replace $\Delta$ with $(1-\varepsilon)\Delta+D$.

Suppose that $f:X'\to X$ is any log resolution of $(X,\Delta)$, and write
$$K_{X'}=f^*(K_X+\Delta)+F-F',$$
where $F$ is an effective, $f$-exceptional Cartier divisor and $F'$ is a simple normal crossing divisor with $\floor{F'}=0$. Then $F\sim_{\Q} K_{X'}-f^*(K_X+\Delta)+F'$. If we pick a divisor $L$ on $X'$ such that $\O_{X'}(L)\cong (f^*\O_{X'}(D))^{DD}$, then by \thref{areyouwrong}, there is an effective $f$-exceptional divisor $E$ such that $L+E\sim_{\Q}f^*D$, and we have
$$F+L+\floor{E}\sim_{\Q}K_{X'}+F'-\{E\}+(f^*D-f^*(K_X+\Delta)).$$

Let us explain why \thref{stronger} can be applied with $\mathcal{L}_{\mathcal{F},f}=\O_{X'}(L+\floor{E})$ to prove that $D$ has rational singularities.

For (i), if we consider an injection $\varphi:f_*\O_{X'}(L)\to f_*\O_{X'}(L+\floor{E})$ which is induced by $\O_{X'}(L)\to \O_{X'}(L+\floor{E})$, $\varphi$ is an isomorphism outside of codimension $\ge 2$ closed subset of $X$. If we let $\mathcal{Q}$ be the cokernel of $\varphi$, then we have an exact sequence
$$ 0\to f_*\O_{X'}(L)\cong \O_X(D)\to f_*\O_{X'}(L+\floor{E})\to \mathcal{Q}\to 0.$$
The first isomorphism is by \thref{goo}. As the proof of \thref{goo}, by taking the local cohomology spectral sequence, we obtain $H^0_x(X,\mathcal{Q})=0$ for a point $x\in X$ with codimension $\ge 2$, and we can prove $\mathcal{Q}=0$.

For (ii), let $E_i$ be the prime $f$-exceptional divisors, and let us define
$$ a_i:=\begin{cases} 1 &\text{ if } \mathrm{mult}_{E_i}(F'-\{E\})<0,\text{ and } \\ 0 & \text{ elsewhere, }\end{cases}$$
and
$$ E':=\sum_i a_iE_i.$$
Then
$$ F+L+\floor{E}+E'\sim_{\Q,f}K_{X'}+F'-\{E\}+E',$$
and $\floor{F'-\{E\}+E'}=0$. Therefore, by the relative Kawamata-Viehweg vanishing theorem (Theorem 3.3.4 in \cite{Fujino}), $R^if_*\O_{X'}(F+L+\floor{E}+E')=0$. If we consider the injection $\O_{X'}(L+\floor{E})\to \O_{X'}(F+L+\floor{E}+E')$, then we have $R^if_*\O_{X'}(L+\floor{E})=0$ by \thref{twospectral}.

For (iii), by the relative Kawamata-Viehweg vanishing theorem (Theorem 3.3.4 in \cite{Fujino}), 
$$R^if_*\O_{X'}(K_{X'}-(L+\floor{E}))=R^if_*\O_{X'}(K_{X'}-f^*D+\{E\})=0 \text{ 
for }i>0.$$

Thus, we achieved all the conditions listed in \thref{stronger}, so we deduce the assertion.
\end{proof}


\begin{proof}[Proof of \thref{iex11}]
Since having rational singularities is a local property, we may assume $X:=\Spec R$ is an affine scheme and $rD\sim 0$. Moreover, by multiplying some integer on $r$, we may take $r=p^e-1$ for sufficiently large $e$. Let $x\in X$ be a point of $X$ and $f:X'\to X$ any resolution. We follow the proof of Theorem 3.1 in \cite{PS14}.

We may equip an endomorphism $F^e_*:H^{\codim_X x}_x(X,\O_X(D))\to H^{\codim_X x}_x(X,\O_X(D))$ on $H^{\codim_X x}_x(X,\O_X(D))$ as an endomorphism of an abelian group. Indeed, if we consider the inclusion $\O_X\subseteq F^e_*\O_X$, then by tensoring $\O_X(D)$ and reflexing, we have a map $\O_X(D)\to F^e_*\O_X(p^eD)$. Since $(p^e-1)D\sim 0$ holds, we obtain a map $\O_X(D)\to F^e_*\O_X(D)$. Taking local cohomology gives the desired map. Let us call a submodule $K\subseteq H^{\codim_X x}_x(X,\O_X(D))$ \emph{$F$-stable} if $F^e_*(K)\subseteq K$.

Let $R^{\wedge}$ be the completion of $R$ along $x$, $K\subsetneq H^{\codim_X x}_{x}(X,\O_X(D))$ any $F$-stable submodule of $H^{\codim_X x}_x(X,\O_X(D))$. It is worth noting that $H^{\codim_X x}_x(X,\O_X(D))$, and its submodules are also $R^{\wedge}$-modules. 

Suppose that $c\in \mathrm{Ann}_{R^{\wedge}} K$ is any nontrivial element of $\mathrm{Ann}_{R^{\wedge}} K$. Note that such an element exists because of the Matlis duality. Indeed, if we let $N^{\wedge}:=R^{\wedge}N$ for any $R$-module $N$, and $N'$ is the $R$-module corresponding to $\O_X(D)$, then the inclusion $K\subsetneq H^{\codim_X x}_x(X,\O_X(D))$ gives us a surjection 
$$\mathrm{Hom}_{R^{\wedge}}((N')^{\wedge},\omega_{R^{\wedge}})\to \mathrm{Hom}_{R^{\wedge}}(K,E),$$ where $E$ is the injective hull of the residue field of $R$ (see Theorem 3.5.8 in \cite{BH98}). Since $\omega_{R^{\wedge}},(N')^{\wedge}$ are torsion-free,
$$ 0\ne \mathrm{Ann}_{R^{\wedge}}(\mathrm{Hom}_{R^{\wedge}}(K,E))=\mathrm{Ann}_{R^{\wedge}}K.$$

Note that $R^{\wedge}$ is also strongly $F$-regular. There is a positive integer $e'$ and a map $\varphi:F^{e'}_*R^{\wedge}\to R^{\wedge}$ such that the composition
\begin{equation}\label{st}
R^{\wedge}\to F^{e'}_*R^{\wedge}\overset{F^{e'}_*(\cdot c)}\to F^{e'}_*R^{\wedge}\overset{\varphi}{\to} R^{\wedge}
\end{equation}
is the identity. We may assume $e|e'$ and hence $(p^e-1)|(p^{e'}-1)$. Twisting by $\O_X(D)$ and reflexing on \eqref{st}, we have the following composition which is the identity
$$ (N')^{\wedge}\to F^{e'}_*(N')^{\wedge}\to F^{e'}_*(N')^{\wedge}\to (N')^{\wedge}.$$
Now, taking local cohomology gives us
$$ 
\begin{aligned}
&H^{\codim_X x}_x(X,\O_X(D)) \\
\to & H^{\codim_X x}_x(X,F^{e'}_*\O_X(D)) \\
\overset{F^{e'}_*(\cdot c)}\to & H^{\codim_X x}_x(X,F^{e'}_*\O_X(D)) \\
\to & H^{\codim_X x}_x(X,\O_X(D)).
\end{aligned}
$$
Since $K$ is $F$-stable and $cK=0$ holds, $K=0$.

We may mimic the argument of Smith (see \cite{Smi97}, Theorem 3.1). Let $f:X'\to X$ be a resolution. Now, consider the Leray spectral sequence
\begin{equation}\label{st2}
E^{st}_2=H^s_x(X,R^tf_*(f^*\O_X(D))^{DD})\implies H^{s+t}_{f^{-1}(x)}(X',(f^*\O_X(D))^{DD})	
\end{equation}
and its edge map $\delta_{D}:H^{\codim_X x}_x(X,\O_X(D))\to H^{\codim_X x}_{f^{-1}(x)}(X,(f^*\O_{X}(D))^{DD})$. Let $K_D$ be the kernel of $\delta_D$. If we consider Proposition 1.12 in \cite{Smi97} and set $\psi$ by 
$$\psi:(f^*\O_X(D))^{DD}\to F^e_*(f^*\O_X(p^eD))^{DD},$$
then we have a diagram
$$ 
\begin{tikzcd}
H^{\codim_X x}_{f^{-1}(x)}(X’,(f^*\O_X(D))^{DD}) \ar["F^e_*"]{r}& H^{\codim_X x}_{f^{-1}(x)}(X’,(f^*\O_X(D))^{DD}) \\
H^{\codim_X x}_x(X,\O_X(D)) \ar["F^e_*"']{r}\ar["\delta_D"]{u}& H^{\codim_X x}_x(X,\O_X(D))\ar["\delta_D"']{u}
\end{tikzcd}
$$
by considering the argument in the proof of Theorem 3.1 in \cite{Smi97} almost unchangingly. For any $c\in K_D$, $\delta_D(F^e_*(c))=F^e_*(\delta_D(c))=0$ and hence $K_D$ is $F$-stable. Therefore $K_D=0$. It means $\delta_D$ is an injection.

We may use induction. Suppose that $n\ge 2$ is an integer and assume $R^if_*(f^*\O_X(D))^{DD}=0$ for $1\le i<n-1$. Then by inspecting \eqref{st2}, we obtain an exact sequence
$$ 0\to H^0_x(X,R^{n-1}f_*(f^*\O_X(D))^{DD})\to H^n_x(X,\O_X(D))\overset{\delta_{n,D}}{\to} H^n_{f^{-1}(x)}(X',(f^*\O_X(D))^{DD}).$$
For $n<\codim_X x$, we know $\O_X(D)$ is CM by Theorem 3.2 in \cite{PS14} and hence $\delta_{n,D}$ is an injection. Moreover, we have proven that $\delta_{\codim_X x,D}$ is an injection. Thus, $H^0_x(X,R^nf_*(f^*\O_X(D))^{DD})=0$ and $R^nf_*(f^*\O_X(D))^{DD}$ does not have an associated point as $x$. This argument works for any point $x\in X$ and therefore $R^{n-1}f_*(f^*\O_X(D))^{DD}=0$. By considering \thref{idefi1}, $D$ has weakly rational singularities.

The remaining is to show $R^if_*(f^*\O_X(D))^D=0$ for $i\ge 1$. By \thref{warmup1}, it suffices to show $H^{\codim_X x-i}_{f^{-1}(x)}(X',(f^*\O_X(D))^{DD})=0$ for $i\ge 1$. If we consider \eqref{st2}, then $E^{st}_2=0$ for $s+t<\codim_X x$ and hence $E^{\codim_Xx-i}=0$ and that is the assertion.
\end{proof}

\begin{ex}
Consider klt varieties over fields with positive characteristic.

On the positive side, let $X$ be a normal threefold over a field with characteristic $p>5$. Suppose that there exists an effective $\Q$-divisor $\Delta$ on $X$ such that $(X,\Delta)$ is klt. Referring to Theorem 3 in \cite{BerKo} and following a similar approach as in \thref{iex1}, it can be shown that any $\Q$-Cartier Weil divisor $D$ on $X$ possesses rational singularities.

On the negative side, when the characteristic of $k$ is $3$, there exists a klt $\Q$-factorial threefold $X$ where $X$ lacks rational singularities, as seen in Theorem 1.2 of \cite{Ber}. Additionally, for any prime $p>2$ and any field $k$ of characteristic $p$, an example of a variety exhibiting terminal singularities without rational singularities is given in Corollary 2.2 of \cite{Totaro}.
\end{ex}

Let us prove the following lemma before proving \thref{iquo} and \thref{iquo2}.

\begin{lem}\thlabel{referee}
Let $X,Y$ be normal varieties, $p:Y\to X$ any finite flat morphism, and $\mathcal{F}$ any coherent sheaf on $X$. If $p^*\mathcal{F}$ is CM, then $\mathcal{F}$ is CM.
\end{lem}

\begin{proof}
Since $p$ is finite flat, $p$ is a surjection, and thus $p$ is faithfully flat. By Lemma 00LM in \cite{Stacks}, for any point $x\in X$, any regular sequence of $\mathcal{F}_x$ comes from a regular sequence of $\mathcal{F}_y$, where $p(y)=x$. Thus, we have the assertion.
\end{proof}

We aim to establish that the property of having rational singularities remains stable under finite étale morphisms. This may provide a partial resolution to Remark 2.81 (3) in \cite{Kollar2}.


\begin{proof}[Proof of \thref{iquo}]
By the functoriality of resolution of singularities (see Theorem 3.36 in \cite{Kol07}), for some resolution $f:X'\to X$, there is a resolution $f':Y'\to Y$ such that $f,f'$ fit in the following Cartesian diagram:
$$ 
\begin{tikzcd}
Y'\cong X'\times_X Y \ar["p'"']{d}\ar["f'"]{r}& Y \ar["p"]{d}\\
X' \ar["f"']{r}& X.
\end{tikzcd}
$$
Note that $p'$ is finite {\'e}tale because being finite {\'e}tale is stable under base change (see Lemma 01TS and Lemma 02GO in \cite{Stacks}). Moreover, we have a split injection $(f^*\mathcal{F})^{DD}\to (p')_*(p')^*(f^*\mathcal{F})^{DD}$ by Lemma 3.17 in \cite{KTT+22}.

Note that for any coherent sheaf $\mathcal{G}$ on $Y$, $R^ip_*\mathcal{G}=0$ for $i\ge 1$ by Lemma 02OE in \cite{Stacks}. Now, we consider the following two Grothendieck spectral sequences
$$
\begin{aligned}
{}^1E^{st}_2&=R^sp_*(R^t(f')_*((f')^*(p^*\mathcal{F}))^{DD})\implies R^{s+t}(p\circ f')_*((f')^*(p^*\mathcal{F}))^{DD} \\
{}^2E^{st}_2&=R^sf_*(R^t(p')_*((p')^*(f^*\mathcal{F})^{DD}))\implies R^{s+t}(f\circ p')_*((p')^*(f^*\mathcal{F})^{DD}).
\end{aligned}
$$
Note that 
$$((f')^*(p^*\mathcal{F}))^{DD}=((p\circ f')^*\mathcal{F})^{DD}=((f\circ p')^*\mathcal{F})^{DD}=((p')^*(f^*\mathcal{F}))^{DD}=(p')^*(f^*\mathcal{F})^{DD}.$$
Hence, by inspecting the spectral sequences, we obtain 
$$p_*(R^i(f')_*((f')^*(p^*\mathcal{F}))^{DD})=R^if_*((p')_*((p')^*(f^*\mathcal{F})^{DD}))$$ for any $i\ge 0$.

Suppose that $p^*\mathcal{F}$ has weakly rational singularities so that $R^i(f')_*((f')^*(p^*\mathcal{F}))^{DD}=0$ for $i\ge 1$. Then $R^if_*((p')_*((p')^*(f^*\mathcal{F})^{DD}))=0$ for $i\ge 1$. By the splitting 
\begin{equation}\label{inversetrace}
(f^*\mathcal{F})^{DD}\to (p')_*(p')^*(f^*\mathcal{F})^{DD}\to (f^*\mathcal{F})^{DD}	
\end{equation}
 and taking $R^if_*$ on \eqref{inversetrace}, we also have $R^if_*(f^*\mathcal{F})^{DD}=0$ for $i\ge 1$. Therefore $\mathcal{F}$ has weakly rational singularities. If $p^*\mathcal{F}$ is CM, then $\mathcal{F}$ is CM because of \thref{referee}, and thus if $p^*\mathcal{F}$ has rational singularities, then $\mathcal{F}$ has rational singularities.
\end{proof}

\begin{rmk}
We hope that \thref{iquo} is true for any field $k$ of positive characteristic. The difficulty is that there may be no splitting $(f^*\mathcal{F})^{DD}\to (p')_*(p')^*(f^*\mathcal{F})^{DD}\to (f^*\mathcal{F})^{DD}$ if the degree of $p$ is coprime to the characteristic of $k$. Also we hope that \thref{iquo} is true only when $p$ is finite but not necessarily {\'e}tale.

In \thref{iquo}, if $p$ is a quotient by a finite group which may not be {\'e}tale, then the major problem for proving \thref{iquo} is we cannot ensure the splitting of $(f^*\mathcal{F})^{DD}\to (p')_*(p')^*(f^*\mathcal{F})^{DD}$ if $p'$ is generically finite.
\end{rmk}

Using the notation introduced in \thref{iquo}, suppose that $\mathcal{F}$ is $(KV_{\dim X})$. In this case, the requirement that $p$ be étale can be relaxed to be flat.


\begin{proof}[Proof of \thref{iquo2}]
We borrow the proof of Theorem 1 in \cite{Kov00}. We have proven if $p^*\mathcal{F}$ is CM, then $\mathcal{F}$ is CM in \thref{referee}.

Let
$$ 
\begin{tikzcd}
Y' \ar["f'"]{r}\ar["p'"]{d}& Y \ar["p"]{d}\\
X' \ar["f"]{r}& X
\end{tikzcd}
$$
be a diagram, where $f,f'$ are proper birational morphisms, $p'$ is a proper morphism and $X',Y'$ are smooth.

Consider the diagram
$$ 
\begin{tikzcd}
\mathcal{F} \ar["\alpha"]{r}\ar["\beta"]{d}& Rp_*(p^*\mathcal{F}) \ar["\gamma"]{d}\\
Rf_*(f^*\mathcal{F})^{DD} \ar["\delta"]{r}& R(p\circ f')_*((p\circ f')^*\mathcal{F})^{DD}.
\end{tikzcd}
$$
By Lemma 3.17 in \cite{KTT+22}, there is a left inverse $\alpha'$ of $\alpha$. Moreover, $\gamma$ is a quasi-isomorphism because $p^*\mathcal{F}$ has rational singularities. Thus, $\beta$ has a left inverse $\beta':=\alpha'\circ \gamma^{-1}\circ \delta$. 

Now, we may consider a resolution $f:X'\to X$ which is $(KV_{\dim X})$ about $\mathcal{F}$. If we take $R\mathcal{H}om_{\O_X}(-,\omega^{\bullet}_X)$ on
$$ \mathcal{F}\overset{\beta}{\to} Rf_*(f^*\mathcal{F})^{DD}\overset{\beta'}{\to} \mathcal{F},$$
then we have
$$ \mathcal{F}^D\to f_*(f^*\mathcal{F})^D\to \mathcal{F}^D$$
because $\mathcal{F}$ is CM and $f$ is $(KV_{\dim X})$ about $\mathcal{F}$. Thus $f_*(f^*\mathcal{F})^D\cong \mathcal{F}^D$. If we consider \thref{idefi1}, $\theta_{\mathcal{F},f}$ is a quasi-isomorphism and, therefore, $\mathcal{F}$ has rational singularities.
\end{proof}



We prove \thref{i??} regarding with the notion of $(KV_q)$ and $(RS_q)$, and end the section.


\begin{proof}
Let $f:X'\to X$ be any resolution. By dualizing $\theta_{\mathcal{F},f}$, it suffices to show that $R^if_*(f^*\mathcal{F})^{DD}$ has support of codimension $\ge q+1$ for $i\ge 1$. For any point $x\in X$ with $\codim_X x\le q$, $\mathcal{F}_x=\O_{X,x}$. Hence $(R^if_*(f^*\mathcal{F}))^{DD}_x=(R^if_*\O_X)_x$ for $i\ge 0$ by Corollary 3.8.2 in \cite{Har77}. Thus, by Theorem 1.1 in \cite{chatzistamatiou2}, we have the assertion.
\end{proof}



\section{Notion of $(B_{q+1})$}
The goal of this section is to prove \thref{imain}.



\begin{proof}[Proof of \thref{imain}]
Let $f:X'\to X$ be any resolution of $X$ which is $(RS_q)$ about $\mathcal{F}$. For $(a)\implies (b)$, we consider the following spectral sequence
\begin{equation}\label{2}
 E^{st}_2=H^s_x(X,R^tf_*(f^*\mathcal{F})^{DD})\implies H^{s+t}_{f^{-1}(x)}(X',(f^*\mathcal{F})^{DD}).
\end{equation}

We may use induction. Indeed, let us fix any positive integer $2\le n\le q$ and assume that $R^if_*(f^*\mathcal{F})^{DD}=0$ for $1\le i<n-1$.

Given any point $x\in X$ with $\codim_X x\ge q+1$, since $f$ is $(KV_q)$ about $\mathcal{F}$, and \thref{warmup1} holds, we have $H^i_{f^{-1}(x)}(X',(f^*\mathcal{F})^{DD})=0$ for $i\le q$. By the spectral sequence \eqref{2}, we obtain the following exact sequence
$$ 0\to H^0_x(X,R^{n-1}f_*(f^*\mathcal{F})^{DD})\to H^n_x(X,\mathcal{F})\overset{\alpha_n}{\to} H^n_{f^{-1}(x)}(X',(f^*\mathcal{F})^{DD}).$$
From the above exact sequence, $H^0_x(X,R^{n-1}f_*(f^*\mathcal{F})^{DD})\cong H^n_x(X,\mathcal{F})$ and if we use (a), then $H^0_x(X,R^{n-1}f_*(f^*\mathcal{F})^{DD})=0$ holds. Thus, $R^{n-1}f_*(f^*\mathcal{F})^{DD}$ does not have $x$ as an associated point, and thus any generic point of $\Supp R^{n-1}f_*(f^*\mathcal{F})^{DD}$ has codimension $\ge q$ if there exists. Let us assume $\codim_X x\le q$. Then since $f$ is $(RS_q)$ about $\mathcal{F}$, $(R^if_*(f^*\mathcal{F}))_x=0$ for $i\ge 1$. Hence, we have $R^{n-1}f_*(f^*\mathcal{F})^{DD}=0$.

For $(b)\implies (a)$, since $\theta_{\mathcal{F},f}$ is a quasi-isomorphism at $x\in X$ with $\codim_X x\le q$, we deduce $(R^if_*(f^*\mathcal{F})^D)_x\cong \mathcal{E}xt^{-\codim_X x+i}_{\O_{X,x}}((f^*\mathcal{F})^{DD}_x,\omega^{\bullet}_{X,x})$. Hence the fact that $f$ is $(KV_q)$ about $\mathcal{F}$ and the local duality gives us $H^i_x(X,\mathcal{F})=0$ for $i<\codim_X x$.

If $\codim_X x\ge q+1$, by inspecting the spectral sequence \eqref{2}, $E^{s0}_2=E^{s0}_{\infty}$ for $s\le q$. Considering the edge map $E^{s0}_{\infty}\to E^s=H^s_{f^{-1}(x)}(X',(f^*\mathcal{F})^{DD})$ for such $s$, if we use \thref{warmup1}, then we obtain the assertion. 

For $(a)\iff (c)$, let us consider the following exact sequence
$$ 
\begin{tikzcd}
0\ar{r} & f_*(f^*\mathcal{F})^D \ar["\mathcal{H}^0(\theta_{\mathcal{F},f})"]{r}& \mathcal{F}^D\ar{r}&\mathcal{Q} \ar{r}&0 
\end{tikzcd}
$$
for a coherent sheaf $\mathcal{Q}$ on $X$. Since $f$ is $(RS_q)$ about $\mathcal{F}$, the support of $\mathcal{Q}$ has codimension $\ge q+1$ and hence $H^i_x(X,\mathcal{Q})=0$ for $i>\dim X-q$.

By taking local cohomology on the above exact sequence, we obtain
$$ H^i_x(X,f_*(f^*\mathcal{F})^D)\cong H^i_x(X,\mathcal{F}^D)\text{ for any }\dim X-q<i<q\text{ and any closed point }x\in X.$$
Hence, $\mathcal{F}^D$ is $(B_{q+1})$ if and only if $H^i_x(X,f_*(f^*\mathcal{F})^D)=0$ for any $\dim X-q<i<\dim X$ and any closed point $x\in X$.

Let $x\in X$ be any closed point and consider the Leray spectral sequence
$$ E^{st}_2=H^s_x(X,R^tf_*(f^*\mathcal{F})^D)\implies H^{s+t}_{f^{-1}(x)}(X',(f^*\mathcal{F})^D).$$
Then, by the fact that $f$ is $(KV_q)$ about $\mathcal{F}$ and the dimension counting, $E^{st}_2=0$ for $s+t\ge \dim X-q$. Hence, by inspecting the spectral sequence, $E^{s0}_2=E^{s0}_{\infty}$ for $s>\dim X-q$. Moreover, $E^{s0}_{\infty}=E^s$ for $s>\dim X-q$. Thus, $\mathcal{F}^D$ is $(B_{q+1})$ if and only if $H^i_{f^{-1}(x)}(X',(f^*\mathcal{F})^D)=0$ for any $\dim X-q<i<\dim X$ and any closed point $x\in X$. If we use \thref{warmup1}, then we have the assertion.
\end{proof}

\begin{rmk}
In \thref{imain}, the argument of the proof of $(a)\iff (b)$ is similar to the argument of the proof of Lemma 3.1 in \cite{Alex} and Proposition 7.1.7 in \cite{Fujino}. Moreover, the idea of the proof of $(a)\iff (c)$ is inspired by the proof of Lemma 3.3 in \cite{Kov}.
\end{rmk}



\begin{proof}[Proof of \thref{imaincoro}]
Note that any reflexive sheaf $\mathcal{F}$ on $X$ of rank $1$ is $(RS_q)$; indeed $\mathcal{F}$ is $(KV_q)$ by assumption, $X$ is $(R_q)$, and thus $\mathcal{F}$ is $(RS_q)$ by \thref{i??}. Since $(\O_X(D))^D$ is $(S_q)$ and $\O_X(D)\cong (\O_X(D))^{DD}$ is $(B_q)$ by \thref{imain}, $\O_X(D)$ is CM.
\end{proof}

\section{$q$-birational morphisms}
In this section, we always assume that the ground field $k$ is of characteristic $0$. Let us prove the following lemma which can be regarded as a partial converse of \thref{twospectral}.



\begin{lem}\thlabel{firstapplicationdead}
Let $X$ be any normal variety, $\mathcal{F}$ any reflexive sheaf of rank $1$ on $X$ which is $(S_{q+1})$, $f:X'\to X$ any resolution, and $\varphi:\mathcal{L}\to \mathcal{L}'$ is a morphism of line bundles on $X'$ such that 
\begin{itemize}
    \item[(i)] there is an isomorphism $\mathcal{F}\cong f_*\mathcal{L}$,
    \item[(ii)] $f_*\varphi:f_*\mathcal{L}\to f_*\mathcal{L}'$ is an isomorphism, and
    \item[(iii)] $R^if_*\mathcal{L}^D=0$ for $i\ge 1$.
\end{itemize} 
If $f$ is $q$-birational, then $R^if_*\mathcal{L}'=0$ for $1\le i<q$.
\end{lem}

\begin{proof}
We may use a similar argument as in the proof of $(a)\iff (c)$ in \thref{imain}.

By taking dual on $\varphi$, we obtain an injection $\varphi^D:(\mathcal{L}')^D\to \mathcal{L}^D$. Let $\mathcal{Q}$ be the cokernel of $f_*\varphi^D$. Then we have the following exact sequence
\begin{equation}\label{100}
0\to f_*(\mathcal{L}')^D\to f_*\mathcal{L}^D\to \mathcal{Q}\to 0.
\end{equation}
Since $f$ is $q$-birational, the support of $\mathcal{Q}$ has codimension $\ge q+1$, and $H^i_x(X,\mathcal{Q})=0$ for $i>\dim X-q$.

By taking local cohomology on \eqref{100}, we have isomorphisms
$$ H^i_x(X,f_*(\mathcal{L}')^D)\cong H^i_x(X,f_*\mathcal{L}^D)$$
for any closed point $x\in X$ and any $\dim X-q<i<\dim X$.

Let us prove $H^i_x(X,f_*\mathcal{L}^D)=0$ using the argument of the proof of $(a)\iff (c)$ in \thref{imain}. Consider
$$ E^{st}_2=H^s_x(X,R^tf_*\mathcal{L}^D)\implies H^{s+t}_{f^{-1}(x)}(X',\mathcal{L}^D).$$
By (iii) in the statement, we deduce that $E^{st}_2=0$ for $s+t\ge \dim X-q$. Hence, inspecting the spectral sequence, we obtain $E^{s0}_2=E^{s0}_{\infty}=E^s$ for $s>\dim X-q$. Thus, it suffices to show
$$ H^i_{f^{-1}(x)}(X',\mathcal{L}^D)=0$$
for any $\dim X-q<i<\dim X$, and we can apply \thref{warmup1} to prove the claim.

Consider the Leray spectral sequence
$$ E^{st}_2=H^s_x(X,R^tf_*(\mathcal{L}')^D)\implies H^{s+t}_{f^{-1}(x)}(X',(\mathcal{L}')^D).$$
By inspecting the spectral sequence, we obtain $H^i_{f^{-1}(x)}(X',(\mathcal{L}')^D)=0$ for any $\dim X-q<i<\dim X$ and any closed point $x\in X$. Now, we may apply \thref{warmup1} and obtain the assertion.
\end{proof}

Using \thref{firstapplicationdead}, we now prove \thref{iyouaresecondapplication}.


\begin{proof}[Proof of \thref{iyouaresecondapplication}]
If we define a divisor $L$ on $X'$ such that $\O_{X'}(L)=(f^*\O_X(f_*D))^{DD}$, then there is an effective $f$-exceptional divisor $E$ such that $f^*f_*D\sim_{\Q} L+E$ by \thref{areyouwrong} and hence $D-(L+E)\sim_{\Q}D-f^*f_*D$ is anti $f$-nef. Thus, $D-(L+E)$ and $D-(L+\floor{E})=D-(L+E)+\{E\}$ are effective $f$-exceptional divisors on $X'$ by the Negativity lemma.

In order to apply \thref{firstapplicationdead}, put $\mathcal{F}:=\O_X(f_*D)$, $\mathcal{L}:=\O_{X'}(L+\floor{E})$ and $\mathcal{L}'=\O_{X'}(D)$. Then there is an injection $\varphi:\mathcal{L}\to \mathcal{L}'$.

(i) in \thref{firstapplicationdead} is satisfied trivially. In the proof of \thref{iex1}, we proved 
$$f_*\O_X(L+\floor{E})=f_*\O_X(L),$$ and by applying \thref{goo}, we deduce that $f_*\O_X(L)=\O_X(f_*D)$. Thus, we obtain (ii) in \thref{firstapplicationdead}. For (iii) in \thref{firstapplicationdead}, as in the proof of \thref{iex1}, we can prove $R^if_*\O_{X'}(L+\floor{E})=0$ for $i\ge 1$. Thus, we can apply \thref{firstapplicationdead} to prove the assertion.
\end{proof}

We can write \thref{iyouaresecondapplication} as an absolute cohomology vanishing. For a similar result, see Corollary 3.8 in \cite{trash}.

\begin{coro}\thlabel{thirdapplication}
Let $X,X'$ be any normal projective varieties, $X'$ smooth and $f:X'\to X$ any $q$-birational morphism. Suppose that $D$ is any anti $f$-nef Cartier divisor on $X'$ such that $f_*D$ is a $\Q$-divisor and $(S_{q+1})$ on $X$ and $E$ is any effective $f$-exceptional Cartier divisor on $X'$. Then
$$ H^i(X',\O_{X'}(D+E))=H^i(X,\O_X(f_*D))=H^i(X',\O_{X'}(D))$$
for $0\le i<q$. 
\end{coro}

\begin{proof}
We may use the Leray spectral sequence
$$ E^{st}_2=H^s(X,R^tf_*\O_{X'}(D+E))\implies H^{s+t}(X',\O_{X'}(D+E))$$
and hence $H^i(X,f_*\O_{X'}(D+E))=H^i(X',\O_{X'}(D+E))$ for $0\le i<q$ by \thref{iyouaresecondapplication}.

Let $\O_{X'}(D)\to \O_{X'}(D+E)$ be an injection. Then by taking $f_*$, and applying \thref{warmup2}, we obtain an injection
$$ \varphi:\O_{X}(f_*D)\to f_*\O_{X'}(D+E).$$
Let us use the same argument as in the proof of \thref{goo}. Let $\mathcal{Q}$ be the cokernel of $\varphi$. Then we obtain an exact sequence
$$ 0\to \O_{X}(f_*D)\to f_*\O_{X'}(D+E)\to \mathcal{Q}\to 0.$$
Suppose $x\in X$ is a point such that $\codim_X x\ge 2$ holds. If we combine the fact that $\O_{X}(f_*D)$ is a reflexive sheaf on $X$ with the long exact sequence, then we deduce that $H^0_x(X,\mathcal{Q})=0$. Since $\mathcal{Q}$ has the support of the codimension $\ge 2$, we have that $\mathcal{Q}=0$, and thus we obtain that $\varphi$ is an isomorphism.

Therefore $H^i(X,f_*\O_{X'}(D+E))=H^i(X,\O_X(f_*D))$, and we obtain the assertion.
\end{proof}

\end{document}